\documentclass[12pt,thmsa]{article}
\usepackage{amssymb}
\usepackage{amsfonts}
\usepackage{amsmath}
\usepackage[top=3.8cm,bottom=3.8cm,textwidth=16cm,centering,a4paper]{geometry}
\usepackage{diagbox}
\usepackage{slashbox}

\newtheorem{theorem}{Theorem}[section]

\newtheorem{corollary}[theorem]{Corollary}
\newtheorem{remark}[theorem]{Remark}

\newtheorem{lemma}[theorem]{Lemma}

\newtheorem{definition}[theorem]{Definition}
\newenvironment{proof}[1][Proof]{\noindent\textbf{#1.} }{\ \rule{0.5em}{0.5em}}

\begin{document}

\title{Three bound states with prescribed angular momentum to the
cubic-quintic NLS equations in $\mathbb{R}^{3}$}
\date{{\footnotesize \textit{Dedicated to the memory of Professor Tsung-Fang
Wu}}}
\author{Shuai Yao$^{a}$\thanks{%
E-mail address: shyao@sdut.edu.cn (S. Yao)}, Juntao Sun$^{a}$\thanks{%
E-mail address: jtsun@sdut.edu.cn (J. Sun)} \\
{\footnotesize $^{a}$\emph{School of Mathematics and Statistics, Shandong
University of Technology, Zibo 255049, PR China}}}
\maketitle

\begin{abstract}
In this paper, we investigate bound states with prescribed angular momentum
and mass for the nonlinear Schr\"{o}dinger equations (NLS) with the
cubic-quintic nonlinearity in dimensions three. We demonstrate that there
exist three solutions for the double constrained problem: a local minimizer,
a mountain pass type solution, and a global minimizer. Moreover, by means of
the minimax method, we construct a new mountain pass path and further obtain
the geometric link among the three solutions as well as a comparison of
their energy levels. This seems to be the first paper concerning three
solutions, with the method also being applicable to the single constraint
problem.
\end{abstract}

\textbf{Keywords:} Prescribed angular momentum; Multiple solutions; Orbital
stability; Variational methods.

\textbf{2020 Mathematics Subject Classification:} 35A15; 35Q41; 35B35.

\section{Introduction}

Consider the following the cubic-quintic NLS equations
\begin{equation}
\left\{
\begin{array}{ll}
i\partial _{t}\psi =H\psi -|\psi |^{2}\psi +\mu |\psi |^{4}\psi , & \text{
in }\mathbb{R\times R}^{3}, \\
\psi |_{t=0}=\psi _{0}, &
\end{array}%
\right.   \label{e1-1}
\end{equation}%
where the constants $\mu >0$ characterize the interaction strengths, and
\begin{equation*}
H=-\frac{1}{2}\Delta +V(x).
\end{equation*}%
Here, the potential $V(x)$ is supposed to be trapping, a standard example
being the polyharmonic potential given by $V(x)=\omega |x|^{k}$ with $k>2$,
where $\omega >0$ is the trapping frequency. The cubic-quintic Schr\"{o}%
dinger equation has many interpretations in physics. In the context of a
Boson gas, it describes two-body attractive and three-body repulsive
interactions. The cubic-quintic nonlinearity was introduced in \cite{PPT}
and is employed in numerous physical models \cite{M}. For three dimensional,
it models the evolution of a monochromatic wave complex envelope in a medium
with weakly saturating nonlinearity.

Our working space $\Sigma $ is defined as
\begin{equation*}
\Sigma :=\left\{ u\in H^{1}(\mathbb{R}^{3},\mathbb{R}):\int_{\mathbb{R}%
^{d}}|x|^{k}|u|^{2}dx<+\infty \right\} ,
\end{equation*}%
which is a Hilbert space with the inner product and norm
\begin{equation*}
(u,v)_{\Sigma }:=\int_{\mathbb{R}^{d}}(\frac{1}{2}\nabla u\nabla v+\omega
|x|^{k}uv+uv)dx\quad \text{and }\Vert u\Vert _{\Sigma }^{2}:=\Vert u\Vert _{%
\dot{\Sigma}}^{2}+\Vert u\Vert _{2}^{2},
\end{equation*}%
where $\Vert u\Vert _{\dot{\Sigma}}^{2}:=\frac{1}{2}\Vert \nabla u\Vert
_{2}^{2}+\omega \Vert x^{k/2}u\Vert _{2}^{2}$.

When the local well-posedness holds, then there exist $T_{\max}>0$
and a unique solution $\psi\in C((-T_{\max}, T_{\max}), \Sigma)$ to Eq. (\ref%
{e1-1}) conserving the mass
\begin{equation*}
M(\psi):=\|\psi\|_{2}^{2}
\end{equation*}
and the total energy
\begin{equation*}
E(\psi):=\int_{\mathbb{R}^{3}}\left(\frac{1}{2}|\nabla
\psi|^{2}+\omega|x|^{k}|\psi|^{2}-\frac{1}{2}|\psi|^{4}+\frac{\mu}{3}%
|\psi|^{6}\right)dx.
\end{equation*}
Indeed, besides the mass and total energy, another important physical
quantity exists, namely the mean angular momentum of $\psi$ around a
specified rotation axis in $\mathbb{R}^{3}$. To elaborate further and
without loss of generality, we denote the coordinates of a point in $\mathbb{%
R}^{3}$ by $(x_{1},x_{2},z)$ and assume that the rotation axis is the $z$%
-axis. Consequently, the mean angular momentum of $\psi$ is expressed as
\begin{equation*}
L(u):=\left\langle\psi,L_{z}\psi\right\rangle_{L^{2}(\mathbb{R}^{2})},
\end{equation*}
where
\begin{equation*}
L_{z}u=-i(x_{1}\partial_{x_{2}}\psi-x_{2}\partial_{x_{1}}\psi).
\end{equation*}
In quantum mechanics \cite{B}, the angular momentum of a particle
corresponds to the operator $\mathbb{L}:=-ix\wedge\nabla$, so $L_{z}$ is the
third component of the quantum mechanical angular momentum operator $\mathbb{%
L}$. A straightforward calculation (see \cite{AM}) demonstrates that the
time-evolution of $L(u)$ under the flow of(\ref{e1-1}) satisfies:
\begin{equation*}
L(\psi(t))+\mathrm{i} \int_{0}^{t} \int_{\mathbb{R}^{3}}|\psi(\tau)|^{2}
L_{z} V d x d \tau=L\left(\psi_{0}\right).
\end{equation*}
Since $V(x)$ is radially symmetric, it holds that $L_{z}V=0$. Thus the
dynamics of (\ref{e1-1}) also satisfies the angular momentum conservation
law
\begin{equation*}
L(\psi(t))=L(\psi_{0}).
\end{equation*}

Let us now consider a given angular velocity $\Omega\in \mathbb{R}$ and
recall that $\Omega L_{z}$ is the generator of time-dependent rotations
about the $z$-axis. Specifically, for any $f\in L^{2}(\mathbb{R}^{3})$, the
action of $L_{z}$ is characterized by
\begin{equation*}
e^{\mathrm{i} t \Omega L_{z}} f(x)=f\left(e^{t \Theta} x\right),
\end{equation*}
where $\Theta$ is the skew symmetric matrix given by
\begin{equation*}
\Theta:=\left(%
\begin{array}{ccc}
0 & \Omega & 0 \\
-\Omega & 0 & 0 \\
0 & 0 & 0%
\end{array}%
\right).
\end{equation*}
It is evident that $\left(e^{t\Omega L_{z}}\right)_{t\in\mathbb{R}}$
constitutes a family of unitary operators
\begin{equation*}
e^{t \Omega L_{z}}: \Sigma \rightarrow \Sigma.
\end{equation*}
Define
\begin{equation*}
v(t, x):=e^{ t \Omega L_{z}}\psi(t, x)=\psi\left(t, e^{t \Theta} x\right).
\end{equation*}
By a straightforward calculation, then $v$ satisfies the equation
\begin{equation}  \label{e1-2}
\left\{
\begin{array}{ll}
i\partial_{t}v=Hv-|v|^{2}v+\mu|v|^{4}v-\Omega L_{z}v, &  \\
v|_{t=0}=v_{0}. &
\end{array}%
\right.
\end{equation}
In particular, it holds that $L(v(t))=L(\psi(t))=L(\psi_{0})$. The equation (%
\ref{e1-2}) appears in the mean-field description of rotating Bose-Einstein
condensates.

A standing wave of Eq. (\ref{e1-2}) with a prescribed angular momentum $l>0$
and mass $m>0$ is a solution having the form $v(x,t)=e^{-i\lambda t}u(x)$
such that $u$ weakly solves
\begin{equation}
\begin{array}{ll}
Hu-\lambda u-\Omega L_{z}u=|u|^{2}u-\mu |u|^{4}u, & \quad \forall x\in
\mathbb{R}^{3}.%
\end{array}
\label{e1-3}
\end{equation}%
and
\begin{equation}
\int_{\mathbb{R}^{3}}|u|^{2}dx=m,\quad \int_{\mathbb{R}^{3}}L_{z}u\bar{u}%
dx=l.  \label{e1-4}
\end{equation}%
In this respect, solutions to problem (\ref{e1-3})-(\ref{e1-4}) can be
obtained by searching critical points of the energy functional
\begin{equation*}
E(u):=\int_{\mathbb{R}^{3}}\left[ \frac{1}{2}|\nabla u|^{2}+\omega
|x|^{k}|u|^{2}-\frac{1}{2}|u|^{4}+\frac{\mu }{3}|u|^{6}\right] dx
\end{equation*}%
on the constraint
\begin{equation*}
S(m,l):=\left\{ u\in \Sigma :M(u)=m,L(u)=l\right\} .
\end{equation*}%
Here $\lambda ,\Omega \in \mathbb{R}$ are unknown, which appear as Lagrange
multipliers associated to the double constraints.

In recent years, for a single mass constraint case, the existence
and multiplicity of normalized solutions have been studied by many scholars,
see \cite{ANS,BC,BCPY,C1,CRY,F,GLY,LY}. However, there has been relatively
little research on the doubly constraint problem (\ref{e1-3})-(\ref{e1-4}).
As far as we know, it seems that this type of problem was first studied in
\cite{NS}. For the single power term, i.e. $|u|^{p-2}u$ with $2<p<2+\frac{4}{%
N}$, the energy functional $E$ restricted on $S(m,l)$ is bounded from below.
In \cite{NS}, they obtained solutions to problem (\ref{e1-3})-(\ref{e1-4})
by searching the global minimizer, and orbital stability of solutions was
established. As a complement of the investigation conducted in \cite{NS},
Gou and Shen \cite{GS} considered the mass supercritical case, and they
proved that there exist two non-radial symmetric solutions, one of which is
local minimizer and the other is mountain pass type solution. It is natural
to ask a question: in the case of combining nonlinear terms, will it have an
impact on the number of solutions? This is also the research motivation and
purpose of this paper. Therefore, here we consider the case of combining
nonlinear terms with the focusing cubic and defocusing quintic, which have
important physical significance.

Now, we present the main results. Set
\begin{equation*}
\alpha _{m,l}:=\inf_{u\in S(m,l)}E(u).
\end{equation*}%
For any given $\rho >0$, define
\begin{equation*}
\beta _{m,l}^{\rho }:=\inf_{u\in S(m,l)\cap B_{\rho }}E(u),
\end{equation*}%
where
\begin{equation}
B_{\rho }=\left\{ u\in \Sigma :\Vert u\Vert _{\dot{\Sigma}}^{2}\leq \rho
\right\} .  \label{e1-5}
\end{equation}

\begin{theorem}
\label{t1} Let $\omega >0$. Then there exist $m^{\ast }>0$ and $\mu _{\ast
}>0$ such that for $m<m^{\ast }$ and $\mu <\mu _{\ast }$, there admit two
solutions $u^{(1)}$ and $u^{(2)}$ for problem (\ref{e1-3})-(\ref{e1-4}),
where $u^{(1)}$ is a local minimizer satisfying $E(u^{(1)})=\beta
_{m,l}^{\rho }$, and $u^{(2)}$ is a global minimizer satisfying $%
E(u^{(2)})=\alpha _{m,l}$. Moreover, if the local well-posedness holds, then the set of global minimizer
\begin{equation*}
\mathcal{A}_{m,l}:=\left\{ u\in S(m,l):E(u)=\alpha _{m,l}\right\}
\end{equation*}%
and the set of local minimizer
\begin{equation*}
\mathcal{B}_{m,l}^{\rho }:=\left\{ u\in S(m,l):E(u)=\beta _{m,l}^{\rho
}\right\}
\end{equation*}%
are both orbitally stable under the flow of (\ref{e1-1}).
\end{theorem}

\begin{remark}
The operator $H$ defined on $C_{0}^{\infty}(\mathbb{R}^{3})$ is essentially
self-adjoint on $L^{2}(\mathbb{R}^{3})$, however, we remain uncertain as to whether Eq. (\ref{e1-1}) is locally well-posed. In \cite{YZ}, for single power term $|u|^{p-2}u$, when $2<p<4+\frac{4}{k}$, the local well-posedness can be obtained by applying
Strichartz estimates.
\end{remark}

\begin{theorem}
\label{t2} Let $\mu<\mu_{\ast}$. Then there exist $m_{\ast}\leq m^{\ast}$
and $\tilde{\omega}>0$ such that for $m<m_{\ast}$ and $\omega<\tilde{\omega}$%
, there admits a mountain pass type solution $u^{(3)}$ for problem (\ref%
{e1-3})-(\ref{e1-4}).
\end{theorem}

\begin{remark}
\label{r1}It is noted that in Theorem \ref{t2}, we need to control the
frequency $\omega $. This condition is intended to construct a special
mountain pass geometric structure, so as to obtain the precise geometric
positions of the three solutions (see Corollary \ref{c} below). If only the
mountain pass solution is obtained, this condition can be omitted, and the
proof is similar to that in \cite{GS}, in which the first local minimum is
selected as the starting point of the mountain pass path.
\end{remark}

\begin{corollary}
\label{c} Let $\omega<\tilde{\omega}$, $m<m_{\ast}$ and $\mu<\mu_{\ast}$,
then there exist three solutions $u^{(i)}$ for $i=1,2,3$, and
\begin{equation*}
E(u^{(2)})<0<E(u^{(1)})<E(u^{(3)}).
\end{equation*}
Moreover, there holds
\begin{equation*}
\|u^{(1)}\|_{\dot{\Sigma}}^{2}<\|u^{(3)}\|_{\dot{\Sigma}}^{2}<\|u^{(2)}\|_{%
\dot{\Sigma}}^{2}.
\end{equation*}
\end{corollary}

Let us present the main difficulties and research methods of this paper.
First, we restrict the energy functional $E$ to the ball $S(m,l)\cap
B_{\rho} $ and consider the local minimization problem $\beta_{m,r}^{\rho}$.
Then, we show that the compactness of any minimizing sequence holds and
local minimizer is attained. Here, it is worth mentioning that the
verification of compactness is complicated due to the constraint of angular
momentum conservation and the presence of critical nonlinear term. In
addition, to ensure that the local minimizer is attained in the interior of
the ball $S(m,l)\cap B_{\rho}$, we need to rule out the case where it lies
on the boundary. However, the presence of critical term makes it impossible
to achieve this merely by controlling the mass, which is different from \cite%
{GS}. With the help of local minimizer, we can find global minimizer with
negative energy level through the method of asymptotic analysis. Finally, we
construct a mountain pass solution by establishing the mountain pass
structure of the functional $E$. This method of constructing mountain pass
solution has two main differences from that in \cite{GS}. $(i)$ To obtain
more information about the solutions, we construct a new mountain pass path,
with the global minimizer serving as the endpoint of this path, that is
\begin{equation*}
\gamma(m,l):=\inf_{g\in\Gamma(m,l)}\max_{0\leq t\leq1}E(g(t)),
\end{equation*}
where
\begin{equation*}
\Gamma(m,l):=\left\{g\in C([0,1],S(m,l)):g(0)=(u^{(2)})_{l_{1}},
g(1)=u^{(2)}\right\}.
\end{equation*}
Here, we can find $l_{1}>0$ such that $(u^{(2)})_{l_{1}}\in S(m,l)\cap
B_{\rho}$. In this way, both the energy level and the geometric link can be clearly demonstrated. $(ii)$ we do not need the
Pohozaev identity to ensure the boundedness of the sequence, so the method
of constructing the sequence is simpler.

\begin{remark}
\label{r2}In this paper, we restrict the potential function to polyharmonic
functions. In fact, the conditions on the potential can be more general, see
below assumption in \cite{NS,GS}.\newline
$(V1)$ The potential $V\in C^{\infty }(\mathbb{R}^{3},\mathbb{R})$ is
assumed to be radially symmetric and confining, i.e. $V(x)\rightarrow
+\infty $ as $|x|\rightarrow \infty $. Moreover, there exists $k>2$ and $R>0$%
, such that for $|x|>R$:
\begin{equation*}
c_{\alpha }\left\langle x\right\rangle ^{k-|\alpha |}\leq |\partial ^{\alpha
}V(x)|\leq C_{\alpha }\left\langle x\right\rangle ^{k-|\alpha |},\,\forall
\alpha \in \mathbb{N}^{3},
\end{equation*}%
where $\left\langle x\right\rangle =(1+|x|^{2})^{1/2}$ and $c_{\alpha
},C_{\alpha }$ are positive constants. Furthermore, the Pohozaev identity is
not required to obtain the mountain pass type solution, which allows the
conditions $\left\langle \nabla V(x),x\right\rangle >0$ in \cite{GS} to be
eliminated.
\end{remark}

The rest of the paper is structured as follows. Some preliminaries for
analysis will be given in Section 2. The existence of bound states will be
established in Sections 3 and 4.

\section{Preliminary}

In what follows, we introduce the well-known Gagliardo-Nirenberg inequality
and Sobolev inequality (see \cite{T,W}).

\begin{lemma}[Gagliardo-Nirenberg inequality]
\label{L2-1} Let $q\in (2, 6)$. Then there exists a sharp constant $\mathcal{%
S}_{q}>0$ such that
\begin{equation}
\Vert u\Vert _{q}\leq \mathcal{S}_{q}^{1/q}\Vert \nabla u\Vert _{2}^{\frac{%
3(q-2)}{2q}}\Vert u\Vert _{2}^{\frac{6-q}{2q}},  \label{e2-1}
\end{equation}%
where $\mathcal{S}_{q}=\frac{q}{2\|U_{q}\|_{2}^{q-2}}$, and $U_{q}$ is the
ground state solution of the following equation
\begin{equation*}
-\Delta U+\frac{6-q}{3(q-2)}U=\frac{4}{3(q-2)}|U|^{q-2}U.
\end{equation*}
\end{lemma}

\begin{lemma}[Sobolev inequality]
\label{L2-2} There exists an optimal constant $S>0$ depending only on
dimension such that
\begin{equation}
\left\Vert u\right\Vert _{6}^{6}\leq S\left\Vert \nabla u\right\Vert
_{2}^{6},\quad \forall u\in D^{1,2}(\mathbb{R}^{3}),  \label{e2-2}
\end{equation}%
where $D^{1,2}(\mathbb{R}^{3})$ denotes the completion of $C_{0}^{\infty }(%
\mathbb{R}^{3})$ with respect to the norm $\left\Vert u\right\Vert
_{D^{1,2}}:=\left\Vert \nabla u\right\Vert _{2}$. It is well known \cite{T}
that the optimal constant is achieved by
\begin{equation*}
U_{\varepsilon ,y}\left( x\right) =3^{\frac{1}{4}}\left( \frac{\varepsilon }{%
\varepsilon ^{2}+\left\vert x-y\right\vert ^{2}}\right) ^{\frac{1}{2}}\quad
\text{for }\varepsilon >0,y\in \mathbb{R}^{3},
\end{equation*}%
which are the only positive classical solutions to the critical Lane-Emden
equation
\begin{equation*}
-\triangle \omega =\omega ^{5},\text{ }\omega >0\quad \text{in }\mathbb{R}%
^{3}.
\end{equation*}
\end{lemma}

\begin{lemma}
(\cite[Lemma 3.1]{Z}) \label{L2-3} For $q\in\left[2,6\right)$, the embedding $%
\Sigma\hookrightarrow L^{q}(\mathbb{R}^{d},\mathbb{R})$ is compact.
\end{lemma}

\begin{lemma}
(\cite[Proposition 3.2]{LY})\label{L2-6} If $u\in \Sigma $ weakly solves Eq. (%
\ref{e1-3}), then $u$ satisfies the Pohozaev identity
\begin{equation*}
Q(u)=\int_{\mathbb{R}^{3}}|\nabla u|^{2}dx-\omega k\int_{\mathbb{R}%
^{3}}|x|^{k}|u|^{2}dx+2\mu \int_{\mathbb{R}^{3}}|u|^{6}dx-\frac{3}{2}\int_{%
\mathbb{R}^{3}}|u|^{4}dx=0.
\end{equation*}
\end{lemma}

\begin{proof}
The Pohozaev identity $Q(u)=0$ can be proved by multiplying Eq. (\ref{e1-3})
by $x\cdot \nabla \bar{u}$ and $\bar{u},$ respectively, and then integrating
over $\mathbb{R}^{3}$. Moreover, we also need the following identity
\begin{equation*}
\text{Re}\int_{\mathbb{R}^{3}}[(\Omega \cdot L)u](x\cdot \nabla \bar{u})dx=-%
\frac{3}{2}\int_{\mathbb{R}^{3}}\bar{u}(\Omega \cdot L)udx.
\end{equation*}%
The proof is complete.
\end{proof}

\begin{definition}
\label{d1}For given $c\in \mathbb{R}$, we say that the functional $%
E|_{S(m,l)}$ satisfies the $(PS)_{c}$ condition if any sequence $\left\{
u_{n}\right\} \subset S(m,l)$ for which
\begin{equation*}
E(u_{n})\rightarrow c\quad \text{and }E^{\prime
}|_{S(m,l)}(u_{n})\rightarrow 0,
\end{equation*}%
has a strongly convergent subsequence in $\Sigma $. In addition, one says
that the functional $E|_{S(m,l)}$ satisfies the $(PSP)_{c}$ condition if any
sequence $\left\{ u_{n}\right\} \subset S(m,l)$ for which
\begin{equation*}
E(u_{n})\rightarrow c,\quad E^{\prime }|_{S(m,l)}(u_{n})\rightarrow 0\text{
and }Q(u_{n})\rightarrow 0,
\end{equation*}%
has a strongly convergent subsequence in $\Sigma $.
\end{definition}

We note that the $(PS)_{c}$ condition is stronger than the $(PSP)_{c}$
condition. Applying directly the abstract deformation result in \cite{IT},
we can get the following deformation lemma. For given $c\in \mathbb{R}$, let
\begin{equation*}
E^{c}:=\left\{u\in S(m,l):E(u)\leq c\right\}
\end{equation*}
and denote by $K^{c}$ the set of critical points of $E|_{S(m,l)}$ at level $%
c\in\mathbb{R}$.

\begin{lemma}
\label{L2-8} Assume that the functional $E|_{S(m,l)}$ satisfies the $%
(PS)_{c} $ condition at some level $c\in \mathbb{R}$, then for any
neighborhood $\mathcal{O}\subset S(m,l)$ of $(K^{c}\mathcal{O}=\emptyset $
if $K^{c}=\emptyset )$ and any $\bar{\varepsilon}>0$, there exists $%
\varepsilon \in (0,\bar{\varepsilon})$ and $\eta \in C\left( [0,1]\times
S(m,l),S_{m}(m,l)\right) $ such that the following properties hold.\newline
$(i)$ $\eta (0,u)=u$ for any $u\in S(m,l);$\newline
$(ii)$ $\eta (t,u)=u$ for any $t\in \lbrack 0,1]$ if $u\in E^{c-\bar{%
\varepsilon}};$\newline
$(iii)$ $t\mapsto I(\eta (t,u))$ is nonincreasing for any $u\in S(m,l);$%
\newline
$(iv)$ $\eta \left( 1,E^{c+\varepsilon }\backslash \mathcal{O}\right)
\subset E^{c-\varepsilon }$ and $\eta \left( 1,E^{c+\varepsilon }\right)
\subset E^{c-\varepsilon }\cup \mathcal{O}.$
\end{lemma}

\section{The proof of Theorem \protect\ref{t1}}

\begin{lemma}
\label{L2-5} For any $\rho >0$, there exist $m^{\ast }>0$ and $\mu ^{\ast
}>0 $ such that for any $m<m^{\ast }$ and $\mu <\mu ^{\ast }$,
\begin{equation}
\inf_{S(m,l)\cap B_{b\rho }}E(u)<\inf_{u\in S(m,l)\cap (B_{\rho }\backslash
B_{a\rho })}E(u),  \label{e2-5}
\end{equation}%
where $B_{\rho }$ is as in (\ref{e1-5}) and $0<b<a<1.$
\end{lemma}

\begin{proof}
By \cite[Lemma 2.2]{NS}, we know that $\text{dim}S(m,l)=\infty $ and $%
S(m,l)\neq \emptyset $ is isometrically isomorphic to a non-empty subset of
\begin{equation*}
\left\{ (c_{n})_{n\in \mathbb{N}_{0}}\subset l^{2}(\mathbb{R}^{3}):\sqrt{n}%
c_{n}\in l^{2}(\mathbb{R}^{3})\right\} ,
\end{equation*}%
where $(c_{n})$ are the coefficients from the angular momentum eigenfunction
decomposition of $u$. The isometric isomorphism ensures the equivalence
between the $\Sigma $-norm and the $l^{2}$-norm. Thus for any $\rho >0$,
there exists $m_{1}(\rho )>0$ such that $S(m,l)\cap B_{\rho }\neq \emptyset $
if $m<m_{1}(\rho )$.

For any $u\in S(m,l)\cap (B_{\rho }\backslash B_{a\rho })$, by (\ref{e2-1}),
we have
\begin{eqnarray}
E(u) &=&\Vert u\Vert _{\dot{\Sigma}}^{2}+\frac{\mu }{3}\int_{\mathbb{R}%
^{3}}|u|^{6}dx-\frac{1}{2}\int_{\mathbb{R}^{3}}|u|^{4}dx  \notag \\
&\geq &\Vert u\Vert _{\dot{\Sigma}}^{2}-\frac{\mathcal{S}_{4}}{2}\Vert
\nabla u\Vert _{2}^{3}\Vert u\Vert _{2}  \notag \\
&\geq &a\rho -\frac{\mathcal{S}_{4}}{2}\sqrt{m}\rho ^{\frac{3}{2}}.
\label{e2-6}
\end{eqnarray}%
On the other hand, for any $u\in S(m,l)\cap B_{b\rho }$, we deduce from (\ref%
{e2-2}) that
\begin{equation}
E(u)\leq \Vert u\Vert _{\dot{\Sigma}}^{2}+\frac{\mu S}{3}\Vert \nabla u\Vert
_{2}^{6}\leq b\rho +\frac{2\mu S}{3}(b\rho )^{3}.  \label{e2-4}
\end{equation}%
Thus by (\ref{e2-6}) and (\ref{e2-4}), there exists $m_{2}(\rho )>0$ and $%
\mu ^{\ast }>0$ such that
\begin{equation*}
\frac{\mathcal{S}_{4}}{2}\sqrt{m}\rho ^{\frac{3}{2}}+\frac{2\mu S}{3}(b\rho
)^{3}<(a-b)\rho
\end{equation*}%
for $m<m_{2}(\rho )$ and $\mu <\mu ^{\ast }$. So for $m<m^{\ast }\leq \min
\left\{ m_{1},m_{2}\right\} $ and $\mu <\mu ^{\ast }$, we have (\ref{e2-5})
holds. The proof is complete.
\end{proof}

\begin{lemma}
\label{L3-0} Let $m<m^{\ast }$ and $\mu <\mu ^{\ast }$. Then there exists a
local minimizer $u^{(1)}\in S(m,l)$ such that
\begin{equation*}
E(u^{(1)})=\beta _{m,l}^{\rho }=\inf_{u\in S(m,l)\cap B_{\rho }}E(u).
\end{equation*}
\end{lemma}

\begin{proof}
Let $\left\{ u_{n}\right\} \subset S(m,l)\cap B_{\rho }$ be a minimizing
sequence for $\beta _{m,l}^{\rho }$. Then $\left\{ u_{n}\right\} $ is
bounded in $\Sigma $. By the compactness of the embedding $\Sigma
\hookrightarrow L^{q}(\mathbb{R}^{3})$ for $q\in \lbrack 2,6)$, there exists
$u^{(1)}\in \Sigma $ such that%
\begin{equation*}
\begin{array}{ll}
u_{n}\rightharpoonup u^{(1)} & \text{ in }\Sigma , \\
u_{n}\rightarrow u^{(1)} & \text{ in }L^{q}(\mathbb{R}^{3}), \\
u_{n}\rightarrow u^{(1)} & \text{ a.e. in }\mathbb{R}^{3}.%
\end{array}%
\end{equation*}%
So, we have $u^{(1)}\in B_{\rho }$ and $\Vert u^{(1)}\Vert _{2}^{2}=m$.

Next we prove that $L(u^{(1)})=l$. Using the similar argument in \cite[%
Proposition 2.4]{NS}. Let
\begin{equation*}
L(u_{n})=-i(A-B),
\end{equation*}%
where
\begin{equation*}
A=\left\langle u_{n},x_{1}\partial _{x_{2}}u_{n}\right\rangle _{L^{2}}\text{
and }B=\left\langle u_{n},x_{2}\partial _{x_{1}}u_{n}\right\rangle _{L^{2}}.
\end{equation*}%
By Cauchy-Schwarz inequality, for some $R>0$, we have
\begin{eqnarray*}
\left\vert \int_{|x|>R}\overline{u_{n}}x_{1}\partial
_{x_{2}}u_{n}dx\right\vert &\leq &\Vert \nabla u_{n}\Vert _{2}\left(
\int_{|x|>R}|x|^{2}|u_{n}|^{2}dx\right) ^{1/2} \\
&\leq &\Vert \nabla u_{n}\Vert _{2}\left( \int_{|x|>R}\frac{|x|^{k}}{%
|x|^{k-2}}|u_{n}|^{2}dx\right) ^{1/2} \\
&\leq &\frac{1}{R^{\frac{k-2}{2}}}\Vert \nabla u_{n}\Vert _{2}\Vert
|x|^{k/2}u_{n}\Vert _{2} \\
&\leq &\frac{1}{R^{\frac{k-2}{2}}}\Vert u_{n}\Vert _{\Sigma }^{2}.
\end{eqnarray*}%
Let $\varepsilon >0$. Then there exists $R_{\varepsilon }>0$ such that for $%
n\in \mathbb{N}$,
\begin{equation*}
\left\vert \int_{|x|>R_{\varepsilon }}\overline{u_{n}}x_{1}\partial
_{x_{2}}u_{n}dx\right\vert <\varepsilon /3.
\end{equation*}%
For $|x|\leq R_{\varepsilon }$, we know that $u_{n}\rightarrow u^{(1)}$ in $%
L^{2}(\mathbb{R}^{3})$ and $\partial _{x_{2}}u_{n}\rightharpoonup \partial
_{x_{2}}u^{(1)}$ in $L^{2}(\mathbb{R}^{3})$. This implies that
\begin{equation*}
\overline{u_{n}}x_{1}\partial _{x_{2}}u_{n}\rightarrow \overline{u^{(1)}}%
x_{1}\partial _{x_{2}}u^{(1)}\text{ in }L^{1}(B_{R_{\varepsilon }}(0)).
\end{equation*}%
Hence, there exists $N_{\varepsilon }>0$ such that for $n\geq N_{\varepsilon
}$,
\begin{equation*}
\left\vert \int_{|x|\leq R_{\varepsilon }}\left( \overline{u_{n}}%
x_{1}\partial _{x_{2}}u_{n}-\overline{u^{(1)}}x_{1}\partial
_{x_{2}}u^{(1)}\right) dx\right\vert <\varepsilon /3.
\end{equation*}%
For the estimates of $A$, we have
\begin{eqnarray*}
A &=&\int_{\mathbb{R}^{3}}\overline{u_{n}}x_{1}\partial _{x_{2}}u_{n}dx \\
&=&\left\vert \int_{|x|\leq R}\overline{u_{n}}x_{1}\partial
_{x_{2}}u_{n}dx\right\vert +\left\vert \int_{|x|>R}\overline{u_{n}}%
x_{1}\partial _{x_{2}}u_{n}dx\right\vert \\
&\rightarrow &\left\langle u^{(1)},x_{1}\partial
_{x_{2}}u^{(1)}\right\rangle _{L^{2}},\text{ as }n\rightarrow \infty .
\end{eqnarray*}%
For $B$, we can get the same result by using the above steps. To sum up, we
have
\begin{equation*}
l=\lim_{n\rightarrow \infty }L(u_{n})=-i\left( \left\langle
u^{(1)},x_{1}\partial _{x_{2}}u^{(1)}\right\rangle _{L^{2}}-\left\langle
u^{(1)},x_{2}\partial _{x_{1}}u^{(1)}\right\rangle _{L^{2}}\right)
=L(u^{(1)}).
\end{equation*}%
Moreover, the energy functional $E$ is weakly lower semi-continuous. Hence,
we have
\begin{equation*}
E(u^{(1)})\leq \lim_{n\rightarrow \infty }E(u_{n})=\beta _{m,l}^{\rho }\leq
E(u^{(1)}),
\end{equation*}%
which indicates that%
\begin{equation*}
E(u^{(1)})=\beta _{m,l}^{\rho }\text{ and }u_{n}\rightarrow u^{(1)}\text{ in
}\Sigma .
\end{equation*}%
This implies that any minimizing sequence for $\beta _{m}^{\rho }$ is
precompact and $\mathcal{B}_{m,l}^{\rho }\neq \emptyset $. Moreover, it
follows from Lemma \ref{L2-2} that $u^{(1)}\not\in S(m,l)\cap \partial
B_{\rho }$ as $u^{(1)}\in B_{b\rho }$, where $\partial B_{\rho }:=\left\{
u\in \Sigma :\Vert u\Vert _{\dot{\Sigma}}^{2}=\rho \right\} $. Then $u^{(1)}$
is indeed a critical point of $E|_{S(m,l)}$. Therefore, there exist Lagrange
multipliers $\lambda ,\Omega \in \mathbb{R}$ such that $u^{(1)}$ is a weak
solution to problem (\ref{e1-3})-(\ref{e1-4}). The proof is complete.
\end{proof}

\begin{lemma}
\label{L3-1} Let $\mu >0$. Then the functional $E$ is bounded from below and
coercive on $S(m,l)$ for all $m,l>0.$
\end{lemma}

\begin{proof}
It follows from H\"{o}lder's inequality that
\begin{equation*}
\int_{\mathbb{R}^{3}}|u|^{4}dx\leq \left( \int_{\mathbb{R}%
^{3}}|u|^{2}dx\right) ^{1/2}\left( \int_{\mathbb{R}^{3}}|u|^{6}dx\right)
^{1/2},
\end{equation*}%
leading to
\begin{equation*}
\int_{\mathbb{R}^{3}}|u|^{6}dx\geq m^{-1}\left( \int_{\mathbb{R}%
^{3}}|u|^{4}dx\right) ^{2}.
\end{equation*}%
Then we have
\begin{eqnarray*}
E(u) &=&\Vert u\Vert _{\dot{\Sigma}}^{2}+\frac{\mu }{3}\int_{\mathbb{R}%
^{3}}|u|^{6}dx-\frac{1}{2}\int_{\mathbb{R}^{3}}|u|^{4}dx \\
&\geq &\Vert u\Vert _{\dot{\Sigma}}^{2}+\frac{\mu }{3}m^{-1}\left( \int_{%
\mathbb{R}^{3}}|u|^{4}dx\right) ^{2}-\frac{1}{2}\int_{\mathbb{R}%
^{3}}|u|^{4}dx
\end{eqnarray*}%
which implies that $E$ is bounded from below and coercive on $S(m,l)$ for
all $m,l>0$. Indeed, the quantity
\begin{equation*}
-\frac{1}{2}\int_{\mathbb{R}^{3}}|u|^{4}dx+\frac{\mu }{3}m^{-1}\left( \int_{%
\mathbb{R}^{3}}|u|^{4}dx\right) ^{2}
\end{equation*}%
has a negative lower bound on $S(m,l)$. The proof is complete.
\end{proof}

\begin{lemma}
\label{L3-2}Let $m<m^{\ast }$. Then there exist $0<\mu _{\ast }\leq \mu
^{\ast }$ and $\omega \in S(m,l)$ with $\Vert \omega \Vert _{\dot{\Sigma}%
}>\rho $ such that for $\mu <\mu _{\ast }$, we have $E(\omega )<0$.
Moreover, there holds
\begin{equation*}
\alpha _{m,l}=\inf_{u\in S(m,l)}E(u)<0.
\end{equation*}
\end{lemma}

\begin{proof}
Let us fix $u^{(1)}\in \mathcal{B}_{m,l}^{\rho }$ and $\omega =\tau
^{3/2}u^{(1)}(\tau x)$ for $\tau >0$. Clearly,
\begin{equation*}
\Vert \omega \Vert _{2}^{2}=\Vert u^{(1)}\Vert _{2}^{2}=m\text{ and }%
\left\langle L_{z}\omega ,\omega \right\rangle =\left\langle
L_{z}u^{(1)},u^{(1)}\right\rangle =l.
\end{equation*}%
Then we have $\omega \in S(m,l)\cap (\Sigma \backslash B_{\rho })$ for $\tau
>>1$. For $\mu =0$, we get
\begin{equation*}
E^{0}(\omega )=\frac{\tau ^{2}}{2}\Vert \nabla u^{(1)}\Vert _{2}^{2}+\tau
^{-k}\int_{\mathbb{R}^{3}}|x|^{k}|u^{(1)}|^{2}dx-\frac{|\gamma |\tau ^{3}}{2}%
\int_{\mathbb{R}^{3}}|u^{(1)}|^{4}dx\rightarrow -\infty \quad \text{as }\tau
\rightarrow +\infty ,
\end{equation*}%
where $E^{0}(\omega ):=E(\omega )$ for $\mu =0$. Thus there exists $\omega
\in S(m,l)$ with $\Vert \omega \Vert _{\dot{\Sigma}}>\rho $ such that $%
E^{0}(\omega )<0$. Since $E(\omega )\rightarrow E^{0}(\omega )$ as $\mu
\rightarrow 0^{+}$, there exists $\mu _{\ast }>0$ such that $E(\omega )<0$
for all $\mu <\mu _{\ast }\leq \mu ^{\ast }$. Moreover, we have
\begin{equation*}
\alpha _{m,l}\leq E(\omega )<0.
\end{equation*}%
The proof is complete.
\end{proof}

\begin{lemma}
\label{L3-3}Let $m<m^{\ast }$ and $\mu <\mu _{\ast }$. Then there exists a
global minimizer $u^{(2)}\in S(m,l)$ such that
\begin{equation*}
E(u^{(2)})=\alpha _{m,l}=\inf_{u\in S(m,l)}E(u)<0.
\end{equation*}
\end{lemma}

\begin{proof}
We choose a minimizing sequence $\left\{ u_{n}\right\} \subset S(m,l)$ such
that $\lim_{n\rightarrow \infty }E(u_{n})=\alpha _{m,l}$. It follows from
Lemma \ref{L3-1} that the minimizing sequence $\left\{ u_{n}\right\} $ is
bounded in $\Sigma $. Then, there exists a weakly convergent subsequence,
still denoted by $\left\{ u_{n}\right\} $ such that $u_{n}\rightharpoonup
u^{(2)}$ as $n\rightarrow \infty $ for some $u^{(2)}\in \Sigma $. The
compact embedding of Lemma \ref{L2-3} implies that $u_{n}\rightarrow u^{(2)}$
in $L^{q}(\mathbb{R}^{3})$ for $2\leq q<6$. Thus we have $\Vert u^{(2)}\Vert
_{2}^{2}=m$. By using the similar proof of Lemma \ref{L3-0}, we get that $%
L(u^{(2)})=l$. Moreover, the energy functional $E$ is weakly lower
semi-continuous. Therefore, we have
\begin{equation*}
E(u^{(2)})\leq \lim_{n\rightarrow \infty }E(u_{n})=\alpha _{m,l}\leq
E(u^{(2)}),
\end{equation*}%
which shows that $E(u^{(2)})=\alpha _{m,l}$ and $u_{n}\rightarrow u^{(2)}$
in $\Sigma $. The proof is complete.
\end{proof}

\textbf{Now we are ready to prove Theorem \ref{t1}.} By Lemmas \ref{L3-0}
and \ref{L3-3}, we can get two solutions $u^{(1)},u^{(2)}\in S(m,l)$ of
problem problem (\ref{e1-3})-(\ref{e1-4}) satisfying $\Vert u^{(1)}\Vert _{%
\dot{\Sigma}}<\Vert u^{(2)}\Vert _{\dot{\Sigma}}.$

Next we prove the orbital stability of the sets $\mathcal{A}_{m,l},\mathcal{B%
}_{m,l}^{\rho }$. Since their proofs are similar, we only prove that the set
$\mathcal{B}_{m,l}^{\rho }$ is orbitally stable. Suppose by contradiction
that there exists a constant $\varepsilon _{0}>0$, a sequence of initial
data $\left\{ u_{n}^{0}\right\} \subset \Sigma $, a function $v_{0}\in
\mathcal{B}_{m,l}^{\rho }$ and a sequence $\left\{ t_{n}\right\} \subset
\mathbb{R}^{+}$ such that the unique solution $u_{n}$ of Eq. (\ref{e1-1})
with initial data $u_{n}^{0}$ satisfies
\begin{equation*}
\text{dist}_{\Sigma }(u_{n}^{0},v_{0})<\frac{1}{n}\text{ and dist}_{\Sigma
}(u_{n}(\cdot ,t_{n}),\mathcal{B}_{m,l}^{r})\geq \varepsilon _{0}.
\end{equation*}%
Without loss of generality, we may assume that $\left\{ u_{n}^{0}\right\}
\subset S(m,l)$. Indeed, since $\text{dist}_{\Sigma
}(u_{n}^{0},v_{0})\rightarrow 0$ as $n\rightarrow \infty $, the conservation
laws of the energy and mass imply that
\begin{equation*}
\lim_{n\rightarrow \infty }M(u_{n}^{0})=M(v_{0})\text{ and }%
\lim_{n\rightarrow \infty }E(u_{n}^{0})=E(v_{0}).
\end{equation*}%
Let
\begin{equation*}
|L(u_{n}^{0})-L(v_{0})|=|A_{n}-B_{n}|,
\end{equation*}%
where
\begin{equation*}
A_{n}:=\int_{\mathbb{R}^{3}}(\overline{u_{n}^{0}}x_{1}\partial
_{x_{2}}u_{n}^{0}-\overline{v_{0}}x_{1}\partial _{x_{2}}v_{0})dx
\end{equation*}%
and
\begin{equation*}
B_{n}:=\int_{\mathbb{R}^{3}}(\overline{u_{n}^{0}}x_{2}\partial
_{x_{1}}u_{n}^{0}-\overline{v_{0}}x_{2}\partial _{x_{1}}v_{0})dx.
\end{equation*}%
By using above similar arguments, we conclude that for any $\varepsilon >0$
there exists $R_{\varepsilon }>0$ sufficiently large such that
\begin{equation*}
\left\vert \int_{|x|>R_{\varepsilon }}\overline{u_{n}^{0}}x_{1}\partial
_{x_{2}}u_{n}^{0}dx\right\vert <\frac{\varepsilon }{3}
\end{equation*}%
and
\begin{equation*}
\left\vert \int_{|x|>R_{\varepsilon }}\overline{v_{0}}x_{1}\partial
_{x_{2}}v_{0}\right\vert <\frac{\varepsilon }{3}.
\end{equation*}%
The strong convergence in $\Sigma $ implies that $\overline{u_{n}^{0}}%
\partial _{x_{2}}u_{n}^{0}\rightarrow \overline{v_{0}}\partial _{x_{2}}v_{0}$
in $L^{1}(\mathbb{R}^{3})$. So, we have
\begin{equation*}
\left\vert \int_{|x|\leq R_{\varepsilon }}(\overline{u_{n}^{0}}x_{1}\partial
_{x_{2}}u_{n}^{0}-\overline{v_{0}}x_{1}\partial _{x_{2}}v_{0})dx\right\vert <%
\frac{\varepsilon }{3},
\end{equation*}%
and further
\begin{equation*}
|A_{n}|\leq \left\vert \int_{|x|\leq R_{\varepsilon }}(\overline{u_{n}^{0}}%
x_{1}\partial _{x_{2}}u_{n}^{0}-\overline{v_{0}}x_{1}\partial
_{x_{2}}v_{0})dx\right\vert +\left\vert \int_{|x|>R_{\varepsilon }}\overline{%
u_{n}^{0}}x_{1}\partial _{x_{2}}u_{n}^{0}dx\right\vert +\left\vert
\int_{|x|>R_{\varepsilon }}\overline{v_{0}}x_{1}\partial
_{x_{2}}v_{0}dx\right\vert <\varepsilon
\end{equation*}%
for any $n>N_{\varepsilon }$. Similarly, we have $|B_{n}|<\varepsilon $ for
any $n>N_{\varepsilon }$, which implies that
\begin{equation*}
\lim_{n\rightarrow \infty }L(u_{n}^{0})=L(v_{0})=l.
\end{equation*}%
By conservation laws, we obtain that
\begin{equation*}
\lim_{n\rightarrow \infty }M(u_{n}(\cdot ,t_{n}))=M(v_{0}),\text{ }%
\lim_{n\rightarrow \infty }E(u_{n}(\cdot ,t_{n}))=E(v_{0})
\end{equation*}%
and
\begin{equation*}
\lim_{n\rightarrow \infty }L(u_{n}(\cdot ,t_{n}))=L(v_{0})=l.
\end{equation*}%
We now claim that $u_{n}(\cdot ,t_{n})$ is a minimizing sequence for $\beta
_{m,l}^{\rho }$ provided $u_{n}(\cdot ,t_{n})\subset B_{\rho }$. Indeed, if $%
u_{n}(\cdot ,t_{n})\subset (\Sigma \backslash B_{\rho })$, then by the
continuity there exists $\bar{t}_{n}\in \lbrack 0,t_{n})$ such that $\left\{
u_{n}(\cdot ,\bar{t}_{n})\right\} \subset \partial B_{\rho }$. Thus, there
holds%
\begin{equation*}
E(u_{n}(\cdot ,\bar{t}_{n}))\geq \inf_{u\in S(m,l)\cap \partial B_{\rho
}}E(u)>\inf_{u\in S(m,l)\cap B_{b\rho }}E(u)\geq \inf_{u\in S(m,l)\cap
B_{\rho }}E(u)=\beta _{m,l}^{\rho },
\end{equation*}%
which is a contradiction. Hence, $\left\{ u_{n}(\cdot ,t_{n})\right\} $ is a
minimizing sequence for $\beta _{m,l}^{\rho }$. Therefore, there exists $%
u_{0}\in \mathcal{B}_{m,l}^{\rho }$ such that $u_{n}(\cdot
,t_{n})\rightarrow u_{0}$ in $\Sigma $, which contradicts to
\begin{equation*}
\text{dist}_{\Sigma }(u_{n}(\cdot ,t_{n}),\mathcal{B}_{m,l}^{\rho })\geq
\varepsilon _{0}.
\end{equation*}%
The proof is complete.

\section{The proof of Theorem \protect\ref{t2}}

\begin{lemma}
\label{L4-1} The functional $E|_{S(m,l)}$ satisfies the $(PS)_{c}$ condition
for any $c\neq0$.
\end{lemma}

\begin{proof}
Let $\left\{ u_{n}\right\} \subset S(m,l)$ be a $(PS)_{c}$ sequence of the
functional $E$ with $c\neq 0$. Since $E$ is coercive on $S(m,l)$ by Lemma %
\ref{L3-1}, we obtain that the sequence $\left\{ u_{n}\right\} $ is bounded,
and then there exists $u\in \Sigma $ such that up to a subsequence%
\begin{equation*}
\begin{array}{ll}
u_{n}\rightharpoonup u & \text{ in }\Sigma , \\
u_{n}\rightarrow u & \text{ in }L^{p}(\mathbb{R}^{3})\text{ for }p\in
\lbrack 2,6), \\
u_{n}\rightarrow u & \text{ a.e. in }\mathbb{R}^{3}.%
\end{array}%
\end{equation*}%
Define
\begin{equation*}
F_{1}(u_{n}):=\int_{\mathbb{R}^{3}}|u_{n}|^{2}dx-m\text{ and }%
F_{2}(u_{n}):=\int_{\mathbb{R}^{3}}L_{z}u_{n}\overline{u_{n}}dx-l.
\end{equation*}%
By applying the similar argument in \cite{GS}, we consider the following two
cases. Let $h\in\Sigma$.

Case $(i):F_{1}^{\prime }\left( u_{n}\right) $ and $F_{2}^{\prime }\left(
u_{n}\right) $ are linearly independent. Then we define
\begin{equation*}
\eta _{n}:=h-\left\langle F_{1}^{\prime }\left( u_{n}\right) ,h\right\rangle
e_{1,n}-\left\langle F_{2}^{\prime }\left( u_{n}\right) ,h\right\rangle
e_{2,n}\text{ for }h\in \Sigma ,
\end{equation*}%
where $\left\{ e_{i,n}\right\} \subset \Sigma $ is bounded and $\left\langle
F_{i}^{\prime }\left( u_{n}\right) ,e_{j,n}\right\rangle =\delta _{ij}$ for $%
i,j=1,2$. This implies that $\left\{ \eta _{n}\right\} \subset
T_{u_{n}}S(m,l)$. Since $\left\{ \eta _{n}\right\} \subset \Sigma $ is
bounded, it holds
\begin{equation*}
\left\langle E^{\prime }\left( u_{n}\right) ,\eta _{n}\right\rangle
=o_{n}(1).
\end{equation*}%
Denote
\begin{equation*}
\lambda _{n}:=2\left\langle E^{\prime }\left( u_{n}\right)
,e_{1,n}\right\rangle\, \text{ and }\,\Omega _{n}:=2\left\langle E^{\prime
}\left( u_{n}\right) ,e_{2,n}\right\rangle .
\end{equation*}%
Thus we have
\begin{equation}
-\frac{1}{2}\Delta u_{n}+\omega |x|^{k}u_{n}-\lambda _{n}u_{n}-\Omega
_{n}L_{z}u_{n}=\left\vert u_{n}\right\vert ^{2}u_{n}-\mu \left\vert
u_{n}\right\vert ^{4}u_{n}+o_{n}(1).  \label{e4-1}
\end{equation}%
Since $\left\{ u_{n}\right\} \subset \Sigma $ and $\left\{ e_{i,n}\right\}
\subset \Sigma $ are bounded, we obtain that $\left\{ \lambda _{n}\right\}
,\left\{ \Omega _{n}\right\} \subset \mathbb{R}$ are bounded. It then
follows that there exists $\lambda ,\Omega \in \mathbb{R}$ such that $%
\lambda _{n}\rightarrow \lambda $ and $\Omega _{n}\rightarrow \Omega $ in $%
\mathbb{R}$ as $n\rightarrow \infty $. Moreover, we find that $u\in \Sigma $
solves the equation
\begin{equation}
-\frac{1}{2}\Delta u+\omega |x|^{k}u-\lambda _{n}u-\Omega
_{n}L_{z}u=\left\vert u\right\vert ^{2}u-\mu \left\vert u\right\vert
^{4}u+o_{n}(1).  \label{e4-2}
\end{equation}

Case $(ii):F_{1}^{\prime }\left( u_{n}\right) $ and $F_{2}^{\prime }\left(
u_{n}\right) $ are linearly dependent. Then we have
\begin{equation*}
T_{u_{n}}S(m,l)=\left\{ f\in \Sigma :\left\langle u_{n},f\right\rangle
_{2}=0\right\} .
\end{equation*}%
Define
\begin{equation*}
\eta _{n}:=h-\left\langle F_{1}^{\prime }\left( u_{n}\right) ,h\right\rangle
e_{1,n},
\end{equation*}%
where $\left\{ e_{i,n}\right\} \subset \Sigma $ is bounded and $\left\langle
F_{1}^{\prime }\left( u_{n}\right) ,e_{1,n}\right\rangle =1$. Clearly, $%
\left\{ \eta _{n}\right\} \subset T_{u_{n}}S(m,l)$ is bounded in $\Sigma $.
Now we define
\begin{equation*}
\lambda _{n}:=2\left\langle E^{\prime }\left( u_{n}\right)
,e_{1,n}\right\rangle \,\text{ and }\, \Omega _{n}:=0.
\end{equation*}%
Thus we conclude that $\left\{ u_{n}\right\} \subset \Sigma $ satisfies (\ref%
{e4-1}) with $\Omega _{n}=0$ and $u\in \Sigma $ solves (\ref{e4-2}) with $%
\Omega =0$. Note that $F_{1}\left( u\right) =0$, since $u_{n}\rightarrow u$
in $L^{2}(\mathbb{R}^{3})$ as $n\rightarrow \infty $. Using the similar
argument of Theorem \ref{t1}, we are able to show that $F_{2}\left( u\right)
=0$ and so $u_{n}\rightarrow u$ in $\Sigma $. The proof is complete.
\end{proof}

In the sequel, we proceed to establish the existence of mountain pass type
solutions to problem (\ref{e1-3})-(\ref{e1-4}). To achieve this, we first
present the mountain pass energy level. Let%
\begin{equation*}
(u^{(2)})_{l}:=l^{\frac{3}{2}}u^{(2)}(lx)\ \mbox{for}\ l>0,
\end{equation*}%
where $u^{(2)}$ is the solution of problem (\ref{e1-3})-(\ref{e1-4}%
) obtained in Theorem \ref{t1}. Let
\begin{equation*}
h(l):=\frac{l^{2}}{2}\Vert \nabla u^{(2)}\Vert _{2}^{2}+\omega l^{-k}\Vert
x^{k/2}u^{(2)}\Vert _{2}^{2}.
\end{equation*}%
Clearly, the first derivative of $h(l)$ is the following%
\begin{equation*}
h^{\prime }(l)=l\Vert \nabla u^{(2)}\Vert _{2}^{2}-k\omega l^{-(k+1)}\Vert
x^{k/2}u^{(2)}\Vert _{2}^{2}.
\end{equation*}%
Then we can see that $h^{\prime }(l)<0$ for $0<l<l_{0}$, and $h^{\prime
}(l)>0$ for $l>l_{0}$, where%
\begin{equation*}
l_{0}:=\left( \frac{k\omega \Vert x^{k/2}u^{(2)}\Vert _{2}^{2}}{\Vert \nabla
u^{(2)}\Vert _{2}^{2}}\right) ^{\frac{1}{k+2}}.
\end{equation*}%
It follows that $h(l)$ is decreasing when $0<l<l_{0}$ and is increasing when
$l>l_{0}$, which implies that
\begin{equation*}
\inf_{l>0}h(l)=h(l_{0})=\frac{k+2}{2}\omega ^{\frac{2}{k+2}}\Vert \nabla
u^{(2)}\Vert _{2}^{\frac{2k}{k+2}}\Vert x^{k/2}u^{(2)}\Vert _{2}^{\frac{4}{%
k+2}}.
\end{equation*}%
Thus there exists a $\widetilde{\omega }>0$ such that $\inf_{l>0}h(l)<\rho
/2 $ for $0<\omega <\widetilde{\omega }$,. So we obtain that there exist two
positive constants $l_{\ast }$ and $l^{\ast }$ such that $h(l_{\ast
})=h(l^{\ast })=\rho /2$. Now we can fix a $l_{\ast }<l_{1}<l^{\ast }$ such
that $(u^{(2)})_{l_{1}}\in B_{\rho /2}$.

We introduce a min-max class
\begin{equation*}
\gamma (m,l):=\inf_{g\in \Gamma (m,l)}\max_{0\leq t\leq 1}E(g(t)),
\end{equation*}%
where
\begin{equation*}
\Gamma (m,l):=\left\{ g\in
C([0,1],S(m,l)):g(0)=(u^{(2)})_{l_{1}},g(1)=u^{(2)}\right\} .
\end{equation*}%
Notice that $\Gamma (m,l)\neq \emptyset $, for $%
g(t)=(t+l_{1}(1-t))^{3/2}u^{(2)}(tx+l_{1}(1-t)x)\in \Gamma (m,l)$. It
follows from Lemma \ref{L2-5} that there exists $m_{\ast }\leq m^{\ast }$
such that $E((u^{(2)})_{l_{1}})>0$ for $m<m_{\ast }$. Then we have
\begin{equation*}
\gamma (m,l)>\max \left\{ E((u^{(2)})_{l_{1}}),E(u^{(2)})\right\} >0.
\end{equation*}

\begin{lemma}
\label{L4-4} Let $\mu <\mu _{\ast }$, $m<m_{\ast }$ and $\omega <\tilde{%
\omega}$. Then problem (\ref{e1-3})-(\ref{e1-4}) admits a mountain
pass type solution $u^{(3)}$ at level $\gamma (m,l)$.
\end{lemma}

\begin{proof}
Since $\Vert v\Vert _{\dot{\Sigma}}^{2}<\frac{\rho }{2}$ and $\Vert
u^{(2)}\Vert _{\dot{\Sigma}}^{2}>\rho $, there exists $0<\tau <1$ such that $%
\Vert g(\tau )\Vert _{\dot{\Sigma}}^{2}=\frac{3\rho }{4}$ and thus
\begin{equation*}
\max_{t\in \lbrack 0,1]}E(g(t))\geq E(g(\tau ))\geq \inf_{u\in S(c)\cap
(B_{\rho }\backslash B_{\rho /2})}E(u).
\end{equation*}%
By Lemma \ref{L2-8}, we have
\begin{equation*}
\gamma (m,l)\geq \frac{\rho }{2}-\frac{\mathcal{S}_{4}}{2}\sqrt{m}\rho ^{%
\frac{3}{2}}.
\end{equation*}%
Suppose by contradiction that $K^{c}=\emptyset $. By using Lemma \ref{L2-8}
with $\mathcal{O}=\emptyset $ and $\bar{\varepsilon}:=\frac{\rho }{4}-\frac{%
\mathcal{S}_{4}}{2}\sqrt{m}\rho ^{\frac{3}{2}}-\frac{\mu S}{96}\rho ^{3}>0$,
there exists $\varepsilon \in (0,\bar{\varepsilon})$ and $\eta \in
C([0,1]\times S(m,l),S(m,l))$ such that
\begin{equation*}
\eta (1,E^{\gamma (m,l)+\varepsilon })\subset E^{\gamma (m,l)-\varepsilon }.
\end{equation*}%
According to the definition of $\gamma (m,l)$, we can choose $g\in \Gamma
(m,l)$ such that
\begin{equation*}
\max_{t\in \lbrack 0,1]}E(g(t))\leq \gamma (m,l)+\varepsilon .
\end{equation*}

Let
\begin{equation*}
\tilde{g}(t):=\eta (1,g(t))
\end{equation*}%
for $t\in \lbrack 0,1]$. Since
\begin{equation*}
\max \left\{ E(g(0)),E(g(1))\right\} <\frac{\rho }{4}+\frac{\mu S}{96}\rho
^{3}\leq \gamma (m,l)-\varepsilon ,
\end{equation*}%
it follows from Lemma \ref{L2-8} $(ii)$ that $\tilde{g}\in \Gamma (m,l)$.
Thus we have
\begin{equation*}
\gamma (m,l)\leq \max_{t\in \lbrack 0,1]}E(\tilde{g}(t))\leq \gamma
(m,l)-\varepsilon ,
\end{equation*}%
which is a contradiction. The proof is complete.
\end{proof}

\textbf{We give the proof of Theorem \ref{t2}.} It follows directly from
Lemma \ref{L4-4}.

\section{Acknowledgments}

J. Sun was supported by the National Natural Science Foundation of China
(Grant No. 12371174) and Shandong Provincial Natural Science Foundation
(Grant No. ZR2020JQ01).

\end{document}